\documentclass[oneside,english]{amsart}
\usepackage[T1]{fontenc}
\usepackage[latin9]{inputenc}
\usepackage{units}
\usepackage{amstext}
\usepackage{amsthm}
\usepackage{amssymb}

\makeatletter
\numberwithin{equation}{section}
\numberwithin{figure}{section}
\numberwithin{table}{section}
\theoremstyle{plain}
\newtheorem{thm}{\protect\theoremname}[section]
\theoremstyle{definition}
\newtheorem{defn}[thm]{\protect\definitionname}
\theoremstyle{plain}
\newtheorem{lem}[thm]{\protect\lemmaname}

\makeatother

\usepackage{babel}
\providecommand{\definitionname}{Definition}
\providecommand{\lemmaname}{Lemma}
\providecommand{\theoremname}{Theorem}

\begin{document}
\title{POLYNOMIAL EXTENSION OF STRONGER CENTRAL SET THEOREM NEAR ZERO}
\author{Sujan Pal, Anik Pramanick}
\email{sujan2016pal@gmail.com}
\address{Department of Mathematics, University of Kalyani, Kalyani, Nadia-741235,
West Bengal, India}
\email{pramanick.anik@gmail.com}
\address{Department of Mathematics, University of Kalyani, Kalyani, Nadia-741235,
West Bengal, India}
\begin{abstract}
Furstenberg introduced the notion of Central sets in 1981. Later in
1990 V. Bergelson and N. Hindman proved a different but an equivalent
version of the central set theorem. In 2008 D. De, N. Hindman and
D. Strauss proved a stronger version of central sets theorem. Hindman
and Leader first introduced the concept of near zero. Recently S.
Goswami, L. Baglini and S. Patra did a polynomial extension of Stronger
Central sets theorem in 2023. In this article, we proved the Polynomial
stronger Central Sets Theorem near zero.
\end{abstract}

\maketitle

\section{Introduction}

In \cite{key-7} Furstenberg defined a central subset of $\mathbb{N}$
in terms of notions from topological dynamics. (Specifically, a subset
$A$ of $\mathbb{N}$ is central if and only if there exist a dynamical
system $(X,T)$, points $x$ and $y$ of $X$, and a neighborhood
$U$ of $y$ such that $y$ is uniformly recurrent, $x$ and $y$
are proximal, and $A=\left\{ n\in\mathbb{N}:T_{n}\left(x\right)\in U\right\} $.
See \cite{key-7} for the definitions of \textquotedblleft dynamical
system\textquotedblright , \textquotedblleft proximal\textquotedblright ,
and \textquotedblleft uniformly recurrent\textquotedblright ) He showed
that if $\mathbb{N}$ is divided into finitely many classes, then
one of them must be central, and he proved the following theorem. 
\begin{thm}
\label{Classic CST-1} Let $l\in\mathbb{N}$ and for each $i\in\left\{ 1,2,...,l\right\} ,$
let $\left(y_{i,n}\right)_{n=1}^{\infty}$be a $i$-many sequences
in $\mathbb{Z}$. Let $C$ be a central subset of $\mathbb{N}.$ Then
there exists sequences $\left(a_{n}\right)_{n=1}^{\infty}$in $\mathbb{N}$
and $\left(H_{n}\right)_{n=1}^{\infty}$ in $\mathcal{P}_{f}\left(\mathbb{N}\right)$
such that
\end{thm}

$\left(1\right)$ for all $n$, $\max H_{n}<\min H_{n+1}$ and

$\left(2\right)$ for all $F\in\mathcal{P}_{f}\left(\mathbb{N}\right)$
and all $i\in\left\{ 1,2,...,l\right\} ,$
\[
\sum_{n\in F}\left(a_{n}+\sum_{t\in H_{n}}y_{i,t}\right)\in C.
\]

There are many extensions of this Theorem in the literature in different
direction. In one direction Bergelson, Morera and Johnson \cite{key-2}
established the polynomial version of the above Theorem. 
\begin{thm}
\label{p cst}Let G be a countable abelian group, let $j\in\mathbb{N}$
and let $(y_{\alpha})_{\alpha\in\mathcal{F}}$ be an IP-set in $\mathbb{Z}^{j}$.
Let $F\subset P(\mathbb{Z}^{j},\mathbb{Z})$ and let $A\subset\mathbb{Z}$
be a central set in $\mathbb{Z}$ Then there exist an IP-set $(x_{\alpha})_{\alpha\in\mathcal{F}}$
in $\mathbb{Z}$ and a sub-IP-set $(z_{\alpha})_{\alpha\in\mathcal{F}}$
of $(y_{\alpha})_{\alpha\in\mathcal{F}}$ such that
\[
\forall f\in F\qquad\qquad\forall\beta\in\mathcal{F\qquad\qquad}x_{\beta}+f(z_{\beta})\in A.
\]
\end{thm}

Both the Theorems have natural generalization over arbitrary countable
commutative group. In fact Hindman, Maleki and Strauss \cite{key-2}
extended the Theorem \ref{Classic CST-1} for arbitrary semigroup
considering countably many sequenses at a time. Further D. De, N.
Hindman and D. Strauss \cite{key-6-1} extended Theorem \ref{Classic CST-1}
considering arbitrary many sequence at a time. We present here the
commutative version only.
\begin{thm}
\label{st cst}Let $\left(S,+\right)$ be a commutative semigroup
and let $C$ be a central subset of $S$. Then there exist functions
$\alpha:\mathcal{P}_{f}\left(S^{\mathbb{N}}\right)\to S\text{ and }H:\mathcal{P}_{f}\left(S^{\mathbb{N}}\right)\to\mathcal{P}_{f}\left(\mathbb{N}\right)$
such that
\end{thm}

$\left(1\right)$ If $F,G\in\mathcal{P}_{f}\left(S^{\mathbb{N}}\right)$
and $F\subsetneq G$ then $\max H\left(F\right)<\min H\left(G\right)$
and

$\left(2\right)$ If $m\in\mathbb{N},G_{1},G_{2},....,G_{m}\in\mathcal{P}_{f}\left(S^{\mathbb{N}}\right)$;$G_{1}\subsetneq G_{2}\subsetneq....\subsetneq G_{m}$;
and for each $i\in\left\{ 1,2,....,m\right\} ,\left(y_{i,n}\right)\in G_{i},$
then 
\[
\sum_{i=1}^{m}\left(\alpha\left(G_{i}\right)+\sum_{t\in H\left(G_{i}\right)}y_{i,t}\right)\in C.
\]
. 
\begin{proof}
\cite[Theorem 2.2]{key-7}.
\end{proof}
Very natural question arises here that does there exist a polynomial
generalization of Theorem \ref{st cst} in the direction of Theorem
\ref{p cst}. In \cite{key-8} Goswami, Baglini and Patra answered
this question affirmatively.
\begin{thm}
Let $j\in\mathbb{N}$, let $G$ be a countable abelian group and let
$F$ be a finite family of polynomial maps from $G^{j}$ to G such
that $f\left(0\right)=0$ for each $f\in F$. Then for every piecewise
syndetic set $A\subseteq G$ and every IP set $\left(x_{\alpha}\right)_{\alpha\in\mathcal{P}_{f}\left(\mathbb{N}\right)}$,
in $G^{j}$ there exist $a\in A$ and $\beta\in\mathcal{P}_{f}\left(\mathbb{N}\right)$
such that
\[
a+f\left(x_{\beta}\right)\in A
\]
 for all $\beta\in\mathcal{P}_{f}\left(\mathbb{N}\right),f\in F$.
\end{thm}

Another direction of Central Sets Theorem was due to Hindman and Leader
\cite{key-9}. Authors introduced the notion of central set near zero
for dense subsemigroup of $(\mathbb{R},+)$. In fact where as central
sets live at infinity, central sets near zero live near zero and satisfy
conclusion like central sets.
\begin{thm}
\label{Classic CST} Let $S$ be a dense subsemigroup of $(\mathbb{R},+)$
for each $i\in\mathbb{N}$ let $\left(y_{i,n}\right)_{n=1}^{\infty}$
be a sequence in $S$ converging to zero in the usual topology of
$\mathbb{R}$. Let $C$ be a central set near zero in $S$. Then there
exists sequences $\left(a_{n}\right)_{n=1}^{\infty}$ in $S$ converging
to zero in the usual topology of $\mathbb{R}$, and a sequence $\left(H_{n}\right)_{n=1}^{\infty}$
in $\mathcal{P}_{f}\left(\mathbb{N}\right)$ such that
\end{thm}

$\left(1\right)$ for all $n$, $\max H_{n}<\min H_{n+1}$ and

$\left(2\right)$ for all $F\in\mathcal{P}_{f}\left(\mathbb{N}\right)$
and all $i\in\left\{ 1,2,...,l\right\} ,$
\[
\sum_{n\in F}\left(a_{n}+\sum_{t\in H_{n}}y_{i,t}\right)\in C.
\]

\begin{proof}
\cite{key-9}.
\end{proof}
The polynomial generalization of the above Theorem was established
in \cite{key-4}.
\begin{thm}
Let $\left(S,+\right)$ be a dense subsemigroup of $\left(\mathbb{R},+\right)$
containing $0$ such that $\left(S\cap\left(0,1\right),\cdot\right)$
is a subsemigroup of $\left(\left(0,1\right),\cdot\right)$. Let $A$
be a central set near zero in $S$ and $L\in\mathcal{P}_{f}\left(\mathbb{P}\left(S,S\right)\right)$.
Then for any $\delta>0$ , there exist $\left(a_{n}\right)_{n=1}^{\infty}$
in $S$ converging to zero in the usual topology of $\mathbb{R}$,
and a sequence $\left(H_{n}\right)_{n=1}^{\infty}$ in $\mathcal{P}_{f}\left(\mathbb{N}\right)$
such that
\end{thm}

$\left(1\right)$ for all $n$, $\max H_{n}<\min H_{n+1}$ and

$\left(2\right)$ for all $F\in\mathcal{P}_{f}\left(\mathbb{N}\right)$
and all $i\in\left\{ 1,2,...,l\right\} ,$
\[
\left(\sum_{n\in F}a_{n}+P\left(\sum_{n\in F}\sum_{t\in H_{n}}y_{i,t}\right)\right)\in C.
\]

In the present article our aim is to generalize the above theorem
in considering arbitrary many sequences at once.
\begin{thm}
\label{Main} Let $\left(S,+\right)$ be a dense subsemigroup of $\left(\mathbb{R},+\right)$
containing $0$ such that $\left(S\cap\left(0,1\right),\cdot\right)$
is a subsemigroup of $\left(\left(0,1\right),\cdot\right)$. Let $A$
be a $C_{p}^{0}$ set in $S$ and $L\in\mathcal{P}_{f}\left(\mathbb{P}\left(S,S\right)\right)$.
Then for each $\delta\in\left(0,1\right)$, there exist functions
$\alpha_{\delta}:\mathcal{P}_{f}\left(\mathcal{T}_{0}\right)\to S$
and $H_{\delta}:\mathcal{P}_{f}\left(\mathcal{T}_{0}\right)\to\mathcal{P}_{f}\left(\mathbb{N}\right)$
such that 
\begin{enumerate}
\item $\alpha_{\delta}\left(F\right)<\delta$ for each $F\in\mathcal{P}_{f}\left(\mathcal{T}_{0}\right)$,
\item If $F,G\in\mathcal{P}_{f}\left(\mathcal{T}_{0}\right)$ and $F\subset G$,
then $\max H_{\delta}\left(F\right)<\min H_{\delta}\left(G\right)$
and
\item If $n\in\mathbb{N}$ and $G_{1},G_{2},...,G_{n}\in\mathcal{P}_{f}\left(\mathcal{T}_{0}\right),\,G_{1}\subset G_{2}\subset.....\subset G_{n}$
and $f_{i}\in G_{i},i=1,2,...,n.$ then
\[
\sum_{i\in1}^{n}\alpha_{\delta}\left(G_{i}\right)+P\left(\sum_{i\in1}^{n}\sum_{t\in H_{\delta}\left(G_{i}\right)}f_{i}\left(t\right)\right)\in A.
\]
 for all $P\in L$, Where 
\[
\mathcal{T}_{0}=\left\{ \left(x_{n}\right)_{n\in\mathbb{N}}\in S^{\mathbb{N}}:x_{n}\text{ converges to zero in usual topology of }\mathbb{R}\right\} .
\]
.
\end{enumerate}
\end{thm}

\section{Preliminaries}

Let $\left(S,\cdot\right)$ be any discrete semigroup, and $\beta S$
be the set of all ultrafilters on $S$, where the points of $S$ are
identified with the principal ultrafilters. Then $\left\{ \overline{A}:A\subseteq S\right\} $,
where $\overline{A}=\left\{ p\in\beta S:A\in p\right\} $ forms a
closed basis for the toplogy on $\beta S$. With this topology $\beta S$
becomes a compact Hausdorff space in which $S$ is dense, called the
Stone-\v{C}ech compactification of $S$. The operation of $S$ can
be extended to $\beta S$ making $\left(\beta S,\cdot\right)$ a compact,
right topological semigroup with $S$ contained in its topological
center. That is, for all $p\in\beta S$ the function $\rho_{p}:\beta S\to\beta S$
is continuous, where $\rho_{p}(q)=q\cdot p$ and for all $x\in S$,
the function $\lambda_{x}:\beta S\to\beta S$ is continuous, where
$\lambda_{x}(q)=x\cdot q$. For $p,q\in\beta S$ and $A\subseteq S$,
$A\in p\cdot q$ if and only if $\left\{ x\in S:x^{-1}A\in q\right\} \in p$,
where $x^{-1}A=\left\{ y\in S:x\cdot y\in A\right\} $. one can see
\cite{key-3} for an elementary introduction to the semigroup $\left(\beta S,\cdot\right)$
and its combinatorial applications. An element $p\in\beta S$ is called
idempotent if $p\cdot p=p$. A subset $A\subseteq S$ is called central
if and only if $A$ is an element of an idempotent ultrafilter $p$.

Here we will work for those dense subsemigroup $\left(\left(0,1\right),\cdot\right)$,
which are dense subsemigroup of $\left(\left(0,\infty\right),+\right)$,
in this case one can define 
\[
0^{+}=\bigcap_{\epsilon>0}cl_{\beta\left(0,1\right)_{d}}\left(0,\epsilon\right).
\]

$0^{+}$ is two sided ideal of $\left(\beta\left(0,1\right)_{d},\cdot\right)$,
so contains the smallest ideal. It is also a subsemigroup of $\left(\beta\mathbb{R}_{d},+\right)$.
As a compact right topological semigroup, $0^{+}$ has a smallest
two sided ideal. $K\left(0^{+}\right)$ is the smallest ideal contained
in $0^{+}$ . Central Sets near zero are the elements from the idempotent
in $K\left(0^{+}\right)$.

Before we proceed, let us define a shorthand notation to denote polynomials.

For any $m_{1},m_{2},m_{3}\in\mathbb{N}$, and $A\subseteq\mathbb{R}^{m_{1}}$,
$B\subseteq\mathbb{R}^{m_{2}}$ , and $C\subseteq\mathbb{R}^{m_{3}}$
, let $\mathbb{P}_{A}\left(B,C\right)$ be the set of all polynomials
from $B$ to $C$ with zero constant term and coefficients are in
$A$. Whenever $A=\mathbb{N}\cup\left\{ 0\right\} $, we will simply
write $\mathbb{P}\left(B,C\right)$ instead of $\mathbb{P}_{\mathbb{N}\cup\left\{ 0\right\} }\left(B,C\right)$.
We will use the notation $\mathbb{P}$ to denote $\mathbb{P}_{\mathbb{N}\cup\left\{ 0\right\} }\left(\mathbb{N},\mathbb{N}\right)$.

For our work we need to know about $J_{p}$-sets.
\begin{defn}
Let $A\subseteq\mathbb{N}$. Then:
\begin{enumerate}
\item $A$ is called a J-set if and only if for every $H\in\mathcal{P}_{f}\left(\mathbb{N}^{\mathbb{N}}\right)$,
there exists $a\in\mathbb{N}$ and $\beta\in\mathcal{P}_{f}\left(\mathbb{N}\right)$
such that for all $f\in H$, 
\[
a+\sum_{t\in\beta}f\left(t\right)\in A.
\]
 
\item A is called a $J_{p}$-set if and only if for every $F\in\mathcal{P}_{f}\left(\mathbb{P}\right)$,
and every $H\in\mathcal{P}_{f}\left(\mathbb{N}^{\mathbb{N}}\right)$,
there exists $a\in\mathbb{N}$ and $\beta\in\mathcal{P}_{f}\left(\mathbb{N}\right)$
such that for all $p\in F$ and all $f\in H$, 
\[
a+P\left(\sum_{t\in\beta}f\left(t\right)\right)\in A.
\]
\end{enumerate}
\end{defn}

\begin{thm}
\label{min} Let $l,m\in\mathbb{N}$, and $A\subseteq\mathbb{N}$
be a $J_{p}$-set. For each $i\in\left\{ 1,2,....,l\right\} $, let
$\left(x_{\alpha}^{i}\right)_{\alpha\in\mathcal{P}_{f}(\mathbb{N})}$
be an IP-set in $\mathbb{N}$. Then for all finite $F\in\mathcal{P}_{f}\left(\mathbb{P}\right)$,
there exist $a\in\mathbb{N}$ and $\beta\in\mathcal{P}_{f}\left(\mathbb{N}\right)$
such that $\min\beta>m$ and 
\[
a+P\left(x_{\beta}^{i}\right)\in A
\]
 for all $i\in\left\{ 1,2,....,l\right\} $ and $P\in F$.
\end{thm}

\begin{proof}
\cite[Theorem 10]{key-8}.
\end{proof}

\section{Polynomial stronger Central set theorem near zero}
\begin{defn}
Let $\left(S,+\right)$ be a dense subsemigroup of $\left(\mathbb{R},+\right)$
containing $0$ such that $\left(S\cap\left(0,1\right),\cdot\right)$is
a subsemigroup of $\left(\left(0,1\right),\cdot\right)$ and let $A\subseteq S$.
Then $A$ is a J-set near zero if and only if whenever $F\in P_{f}\left(\mathcal{T}_{0}\right)$
and $\delta>0$, there exist $a\in S\cap\left(0,\delta\right)$ and
$H\in P_{f}\left(\mathbb{N}\right)$ such that for each $f\in F$,
\[
a+\sum_{t\in H}f\left(t\right)\in A.
\]
 Where $\mathcal{T}_{0}$ be the set of all sequences in $S$ which
converge to zero.
\end{defn}

Here is the polynomial version of J-sets near zero. We call it $J_{p}$-set
near zero.
\begin{defn}
\label{JP set near zero} Let $\left(S,+\right)$ be a dense subsemigroup
of $\left(\mathbb{R},+\right)$ containing $0$ such that $\left(S\cap\left(0,1\right),\cdot\right)$is
a subsemigroup of $\left(\left(0,1\right),\cdot\right)$. Let $F$
be the set of finite polynomials each of which vanishes at 0. A set
$A\subseteq S$ is called a $J_{p}$-set near zero whenever $F\in\mathcal{P}_{f}\left(\mathbb{P}\left(S,S\right)\right),H\in P_{f}\left(\mathcal{T}_{0}\right)$
and $\delta>0$, there exist $a\in S\cap\left(0,\delta\right)$ and
$\beta\in P_{f}\left(\mathbb{N}\right)$ such that for each $f\in H$
and all $P\in F$, 
\[
a+P\left(\sum_{t\in\beta}f\left(t\right)\right)\in A.
\]
\end{defn}

Now we are going to prove the near zero version of theorem \ref{min}.
\begin{lem}
\label{min near zero} Let $\left(S,+\right)$ be a dense subsemigroup
of $\left(\mathbb{R},+\right)$ containing $0$ such that $\left(S\cap\left(0,1\right),\cdot\right)$
is a subsemigroup of $\left(\left(0,1\right),\cdot\right)$. Let $l,m\in\mathbb{N}$,
and $A\subseteq\mathbb{N}$ be a $J_{p}$-set near zero. For each
$F\in\mathcal{P}_{f}\left(\mathbb{P}\left(S,S\right)\right),H\in P_{f}\left(\mathcal{T}_{0}\right)$
and $\delta>0$, there exist $a\in S\cap\left(0,\delta\right)$ and
$\beta\in P_{f}\left(\mathbb{N}\right)$ with $\min\beta>m$ for any
$m\in\mathbb{N}$ 
\[
a+P\left(\sum_{t\in\beta}f\left(t\right)\right)\in A
\]
For each $f\in H$ and for all $P\in F$.
\end{lem}

\begin{proof}
Let $m\in\mathbb{N}$, $F\in\mathcal{P}_{f}\left(\mathbb{P}\left(S,S\right)\right),H\in P_{f}\left(\mathcal{T}_{0}\right)$
and $\delta>0$, for each $h\in H$ define $g_{h}\in P_{f}\left(\mathcal{T}_{0}\right)$
by $g_{h}\left(t\right)=f\left(t+m\right),t\in\mathbb{N}$. For this
$K=\left\{ g_{h}:h\in H\right\} $, pick $a\in S\cap\left(0,\delta\right)$
and $\gamma\in P_{f}\left(\mathbb{N}\right)$ such that 
\[
a+P\left(\sum_{t\in\gamma}g_{h}\left(t\right)\right)\in A
\]
For each $h\in H$ and all $P\in F$. And let $\beta=m+\gamma$.
\end{proof}
Let $\left(S,+\right)$ be a dense subsemigroup of $\left(\mathbb{R},+\right)$
containing $0$ such that $\left(S\cap\left(0,1\right),\cdot\right)$
is a subsemigroup of $\left(\left(0,1\right),\cdot\right)$.. Also
let $\mathbb{P}$ be the set of all polynomials defined on $\mathbb{R}$.
Let $\left\langle S_{i}\right\rangle _{i=1}^{\infty}$ be a sequence
and $\left\{ S_{\alpha}=\sum_{t\in\alpha}S_{i}:\alpha\in\mathcal{P}_{f}\left(\mathbb{N}\right)\right\} $
an IP-set in $S\cap\left(0,1\right)$ and therefore convergent and
consequently $\lim S_{n}=0$.

There is another way for defining the $J_{p}$-set near zero.
\begin{defn}
Let $l,m\in\mathbb{N}$, and $A\subseteq\mathbb{N}$ be a $J_{p}$-set
near zero. For each $i\in\left\{ 1,2,....,l\right\} $, let $\left(x_{\alpha}^{i}\right)_{\alpha\in\mathcal{P}_{f}(\mathbb{N})}$
be an IP-set in $S\cap\left(0,1\right)$. Then for any finite $F\in\mathcal{P}_{f}\left(\mathbb{P}\left(S,S\right)\right)$,
there exist $a\in S$ and $\beta\in\mathcal{P}_{f}\left(\mathbb{N}\right)$
such that
\[
a+P\left(x_{\beta}^{i}\right)\in A
\]
 for all $i\in\left\{ 1,2,....,l\right\} $ and $P\in F$.
\end{defn}

Let us denote by $\mathcal{J}_{P}^{0}$ the set of all ultrafiltars,
all of whose members are $J_{p}$- set near zero, i.e,

\[
\mathcal{J}_{P}^{0}=\left\{ p\in\beta S_{d}:\text{for all }A\in p,A\text{ is a }J_{p}\text{-set near}\text{ zero }\right\} .
\]

We shall denote $E(\mathcal{J}_{p}^{0})$ the set of all idempotents
in $\mathcal{J}_{p}^{0}$. The following theorem show that $\mathcal{J}_{p}^{0}$
is in fact non empty.
\begin{thm}
Let $\left(S,+\right)$ be a dense subsemigroup of $\left(\mathbb{R},+\right)$
containing $0$ such that $\left(S\cap\left(0,1\right),\cdot\right)$
is a subsemigroup of $\left(\left(0,1\right),\cdot\right)$. Then
$\mathcal{J}_{p}^{0}\neq\emptyset$.
\end{thm}

\begin{proof}
Let $F$ be any given system of finite polynomials each of which vanishes
at $0$ and let $\left\{ \left\langle S_{i}^{j}\right\rangle _{i}^{\infty}:j=1,2,\ldots,l\right\} $
such that be a collection of sequences converging to $0$.

Let $\left[q\right]$ be the set of coefficients of the polynomials
in $F$. Let $\left\{ \left[q_{j}\right]:j=1,2,\ldots,l\right\} $
be a sequence of finite sets such that $\left[q\right]=\cup_{j=1}^{l}\left[q_{j}\right]$.

Let $N=N\left(\left[q\right],r,d\right)$ be the P.H.J number and
let 

\[
Q=\left[q\right]^{N}\times\left[q\right]^{N\times N}\times\ldots\times\left[q\right]^{N^{d}}.
\]

There exist $a\in Q$ and $\gamma\in\mathcal{P}_{f}\left(\mathbb{N}\right)$
such that 
\[
\left\{ a\oplus x_{1}\gamma\oplus x_{2}\left(\gamma\times\gamma\right)\oplus\ldots\oplus x_{d}\gamma^{d}:1\leq x_{i}\leq q\right\} 
\]
 is monocromatic \cite[Theorem 4]{key-4}.

Let us define a set 
\[
B=\left\{ \begin{array}{c}
\sum_{i=1}^{N}a_{i}S_{i}^{j}+\sum_{i_{1},i_{2}=1}^{N}a_{i_{1}}a_{i_{2}}S_{i_{1}}^{j}S_{i_{2}}^{j}+\ldots+\sum_{i_{1},\ldots,i_{d}=1}^{N}a_{i_{1}}\ldots a_{i_{d}}S_{i_{1}}^{j}\ldots S_{i_{d}}^{j}:\\
a=\left\langle a_{i}\right\rangle _{i=1}^{N},\left\langle a_{i}a_{j}\right\rangle _{i,j=1,1}^{N.N},\ldots,\left\langle a_{i_{1}}\ldots a_{i_{d}}\right\rangle _{i_{1},\ldots,i_{d}=1,\ldots,1}^{N,\ldots,N}\in Q
\end{array}\right\} 
\]
 where $j\in\left\{ 1,2,\ldots,l\right\} $.

Now define a map $C:Q\to S$ by 
\[
\begin{array}{c}
C\left(\left\langle a_{i}\right\rangle _{i=1}^{N},\left\langle a_{i}a_{j}\right\rangle _{i,j=1,1}^{N.N},\ldots,\left\langle a_{i_{1}}\ldots a_{i_{d}}\right\rangle _{i_{1},\ldots,i_{d}=1,\ldots,1}^{N,\ldots,N}\right)=\\
\sum_{i=1}^{N}a_{i}S_{i}^{j}+\sum_{i_{1},i_{2}=1}^{N}a_{i_{1}}a_{i_{2}}S_{i_{1}}^{j}S_{i_{2}}^{j}+\ldots+\sum_{i_{1},\ldots,i_{d}=1}^{N}a_{i_{1}}\ldots a_{i_{d}}S_{i_{1}}^{j}\ldots S_{i_{d}}^{j}
\end{array}.
\]

where $j\in\left\{ 1,2,\ldots,l\right\} $.

Now 
\[
\begin{array}{c}
C\left(a\oplus x_{1}\gamma\oplus x_{2}\left(\gamma\times\gamma\right)\oplus\ldots\oplus x_{d}\gamma^{d}\right)=b+x_{1}\sum_{t\in\gamma}S_{t}^{j}+x_{2}\left(\sum_{t\in\gamma}S_{t}^{j}\right)^{2}+\ldots+x_{d}\left(\sum_{t\in\gamma}S_{t}^{j}\right)^{d}\\
=b+P\left(\sum_{t\in\gamma}S_{t}^{j}\right)
\end{array}
\]
for all $j\in\left\{ 1,2,\ldots,l\right\} $. Then the following cases
appear 
\begin{enumerate}
\item $x_{1},x_{2},\ldots,x_{d}\in\left[q_{1}\right]$.
\item $x_{1},x_{2},\ldots,x_{d}\in\left[q_{2}\right]$.
\item $\ldots$
\item $x_{1},x_{2},\ldots,x_{d}\in\left[q_{l}\right]$.
\end{enumerate}
For each of the cases, we have $\left\{ b+P\left(\sum_{t\in\gamma}S_{t}^{j}\right):j\in\left\{ 1,2,\ldots,l\right\} ,P\in F\right\} $
is monocromatic.

Then by definition of piecewise syndetic set near zero \cite[Definition 3.4]{key-2-1}
and \cite[Theorem 7]{key-4}, we have if $A$ be an piecewise syndetic
set near zero then $b+P\left(\sum_{t\in\gamma}S_{t}^{j}\right)=b+P\left(S_{\gamma}^{j}\right)\in A$,
for all $j\in\left\{ 1,2,\ldots,l\right\} ,P\in F$.

Since $\overline{A}\cap K\left(0^{+}\left(S\right)\right)\neq\emptyset$,
then there exists $p\in\overline{A}\cap K\left(0^{+}\left(S\right)\right)$
so $p\in\mathcal{J}_{p}^{0}$. So we get our desire result.
\end{proof}
We know that 
\[
J\left(S\right)=\left\{ p\in\beta S:\forall A\in p,A\text{ is a J-set}\right\} 
\]
 is a two sided ideal of $\beta S$.

In \cite[Thorem 3.9]{key-3} Bayatmanesh, Tootkaboni proved that 
\[
J_{0}\left(S\right)=\left\{ p\in\beta S:\forall A\in p,A\text{ is a J-set near zero}\right\} 
\]
 is also a two sided ideal of $0^{+}$.

After defining $J_{p}$-set, Goswami, Baglini and Patra demonstated
that 
\[
\mathcal{J}_{P}=\left\{ p\in\beta\mathbb{N}:\text{for all }A\in p,A\text{ is a }J_{p}\text{-set}\right\} 
\]
 is also a two sided ideal of $\left(\beta\mathbb{N},+\right)$\cite[Thorem 8]{key-8}.

So we expect that $\mathcal{J}_{p}^{0}$ also a two sided ideal of
$0^{+}$. The proof is in the following.
\begin{thm}
\label{Ideal } Let $\left(S,+\right)$ be a dense subsemigroup of
$\left(\mathbb{R},+\right)$ containing $0$ such that $\left(S\cap\left(0,1\right),\cdot\right)$
is a subsemigroup of $\left(\left(0,1\right),\cdot\right)$. Then
$\mathcal{J}_{p}^{0}$ is a two sided ideal of $0^{+}$.
\end{thm}

\begin{proof}
Let $p\in\mathcal{J}_{P}^{0}$ and $q\in0^{+}\left(S\right)$. We
want to show $p+q,q+p\in\mathcal{J}_{P}^{0}$.

To show, $p+q\in\mathcal{J}_{P}^{0}$, let $A\in p+q$, so that $B=\left\{ x\in S:-x+A\in q\right\} \in p$.
Hence $B$ is a $J_{p}$-set near zero. Then for any $F\in\mathcal{P}_{f}\left(\mathbb{P}\left(S,S\right)\right),H\in P_{f}\left(\mathcal{T}_{0}\right)$
and $\delta>0$, there exist $a\in S\cap\left(0,\nicefrac{\delta}{2}\right)$
and $\beta\in P_{f}\left(\mathbb{N}\right)$ such that for each $f\in H$
and all $P\in F$, 
\[
a+P\left(\sum_{t\in\beta}f\left(t\right)\right)\in B.
\]
.

Then 
\[
\bigcap_{f\in F}\left\{ -\left(a+P\left(\sum_{t\in\beta}f\left(t\right)\right)\right)+A\right\} \in q.
\]

Let us choose
\[
y\in\left[\bigcap_{f\in F}\left\{ -\left(a+P\left(\sum_{t\in\beta}f\left(t\right)\right)\right)+A\right\} \right]\cap\left(0,\nicefrac{\delta}{2}\right).
\]

So, for all $f\in F$, and all $P\in F$, 
\[
y+a+P\left(\sum_{t\in\beta}f\left(t\right)\right)\in A,
\]

Let us define, $x=y+a\in S\cap\left(0,\delta\right)$.

Hence for any $F\in\mathcal{P}_{f}\left(\mathbb{P}\left(S,S\right)\right),H\in P_{f}\left(\mathcal{T}_{0}\right)$
and $\delta>0$, there exist $x\in S\cap\left(0,\delta\right)$ and
$\beta\in P_{f}\left(\mathbb{N}\right)$ such that for each $f\in H$
and all $P\in F$, 
\[
x+P\left(\sum_{t\in\beta}f\left(t\right)\right)\in A.
\]

Therefore $A$ is a $J_{p}$-set near zero. Since $A$ is arbitrary
element from $p+q$. So, $p+q\in\mathcal{J}_{P}^{0}$.

Now, If $A\in q+p$, then $B=\left\{ x\in S:-x+A\in p\right\} \in q$.
Choose $x\in B\cap\left(0,\nicefrac{\delta}{2}\right)$, $-x+A\in p$.
So, $-x+A$ is a $J_{p}$-set near zero.

Therefore for any $F\in\mathcal{P}_{f}\left(\mathbb{P}\left(S,S\right)\right),H\in P_{f}\left(\mathcal{T}_{0}\right)$
and $\delta>0$, there exist $a\in S\cap\left(0,\nicefrac{\delta}{2}\right)$
and $\beta\in P_{f}\left(\mathbb{N}\right)$ such that for each $f\in H$
and all $P\in F$, 
\[
a+P\left(\sum_{t\in\beta}f\left(t\right)\right)\in-x+A.
\]

Hence for all $f\in F$, and all $P\in F$, 
\[
x+a+P\left(\sum_{t\in\beta}f\left(t\right)\right)\in A,
\]

Let us define, $z=x+a\in S\cap\left(0,\delta\right)$

So for any $F\in\mathcal{P}_{f}\left(\mathbb{P}\left(S,S\right)\right),H\in P_{f}\left(\mathcal{T}_{0}\right)$
and $\delta>0$, there exist $z\in S\cap\left(0,\delta\right)$ and
$\beta\in P_{f}\left(\mathbb{N}\right)$ such that for each $f\in H$
and all $P\in F$, 
\[
z+P\left(\sum_{t\in\beta}f\left(t\right)\right)\in A.
\]

Therefore $A$ is a $J_{p}$-set near zero. Since $A$ is arbitrary
element from $q+p$. So, $q+p\in\mathcal{J}_{p}^{0}$.
\end{proof}
Since $\mathcal{J}_{P}^{0}$ is a two sided ideal of $0^{+}$. Then
$K\left(0^{+}\right)\subseteq\mathcal{J}_{p}^{0}$.
\begin{lem}
\label{Closed } Let $\left(S,+\right)$ be a dense subsemigroup of
$\left(\mathbb{R},+\right)$ containing $0$ such that $\left(S\cap\left(0,1\right),\cdot\right)$
is a subsemigroup of $\left(\left(0,1\right),\cdot\right)$. Then
$\mathcal{J}_{p}^{0}$ is closed subset of $0^{+}$.
\end{lem}

\begin{proof}
Let $p\in0^{+}\setminus\mathcal{J}_{P}^{0}$, then there exists $A\in p$
such that $A$ is not a $J_{p}$-set. Then $\bar{A}\cap\mathcal{J}_{p}^{0}=\emptyset$
and $p\in\overline{A}$. So $p$ is not a limit point of $\mathcal{J}_{p}^{0}$.
So the proof is done.

Where $\overline{A}=\left\{ p\in\beta S:A\in p\right\} $.
\end{proof}
By Ellis\textquoteright{} theorem \cite[Corollary 2.39]{key-12},
a straightforward consequence of Theorem \ref{Ideal } and lemma \ref{Closed }
is that there are idempotent ultrafilters, and even minimal idempotent
ultrafilters, in $\mathcal{J}_{p}^{0}$ . We denote the set of all
idempotents in $\mathcal{J}_{p}^{0}$ by $E\left(\mathcal{J}_{p}^{0}\right)$.
\begin{defn}
$A\subseteq S$ is said to be a $C_{p}^{0}$ set in $S$ if $A\in p$
for some ultrafilter $p\in E\left(\mathcal{J}_{p}^{0}\right)$.
\end{defn}

The central sets theorem near zero originally proved by Hindman and
Leader in \cite{key-9}. Recently the Central Set Theorem was extended
by Goswami, Baglini and Patra for polynomials in \cite{key-8}.

In \cite{key-8} the authors introduce $C_{p}$-set by defining that,
$C_{p}$-set is the member of the ultrafilters in $E\left(\mathcal{J}_{p}\right)=\left\{ p\in\mathcal{J}_{p}:p\text{ is an idempotent ultrafilter}\right\} $.
\begin{thm}
Let $A$ be a $C_{p}$-set and let $F\in\mathcal{P}_{f}\left(\mathbb{P}\right)$.
There exist functions $\alpha:\mathcal{P}_{f}\left(\mathbb{N}^{\mathbb{N}}\right)\to S\text{ and }H:\mathcal{P}_{f}\left(\mathbb{N}^{\mathbb{N}}\right)\to\mathcal{P}_{f}\left(\mathbb{N}\right)$
such that 
\begin{enumerate}
\item If $G,K\in\mathcal{P}_{f}\left(\mathbb{N}^{\mathbb{N}}\right)$ and
$G\subsetneq K$ then $\max H\left(G\right)<\min H\left(K\right)$
and 
\item If $n\in\mathbb{N},G_{1},G_{2},....,G_{m}\in\mathcal{P}_{f}\left(\mathbb{N}^{\mathbb{N}}\right)$;$G_{1}\subsetneq G_{2}\subsetneq....\subsetneq G_{m}$;
and for each $i\in\left\{ 1,2,....,m\right\} ,f_{i}\in G_{i},$ then
for all $P\in F$,
\[
\sum_{i=1}^{n}\alpha\left(G_{i}\right)+P\left(\sum_{i=1}^{n}\sum_{t\in H\left(G_{i}\right)}f_{i}\left(t\right)\right)\in A.
\]
\end{enumerate}
\end{thm}

\begin{proof}
\cite[Theorem 11]{key-8}.
\end{proof}
Following them we proved the Stronger Polynomial Central sets theorem
near zero$.$ Before that let's recall some notions.
\begin{defn}
Let $\left(S,+\right)$ be a dense subsemigroup of $\left(\mathbb{R},+\right)$
containing $0$ such that $\left(S\cap\left(0,1\right),\cdot\right)$is
a subsemigroup of $\left(\left(0,1\right),\cdot\right)$.
\begin{enumerate}
\item $\mathcal{J}_{p}^{0}=\left\{ p\in\beta S_{d}:\text{for all }A\in p,A\text{ is a }J_{p}\text{ set near }\text{zero }\right\} $.
\item $E\left(\mathcal{J}_{p}^{0}\right)=\left\{ p\in\mathcal{J}_{p}^{0}:p\text{ is an idempotent ultrafilter}\right\} $.
\item A set $A\subseteq S$ is called $C_{p}^{0}$ set if $A\in p\in E\left(\mathcal{J}_{p}^{0}\right).$
\end{enumerate}
\end{defn}

We are now ready to prove our main result, Theorem \ref{Main}. Since
$K\left(0^{+}\right)\subseteq\mathcal{J}_{p}^{0}$, every Central
set near zero is also a $C_{p}^{0}$ set.
\begin{thm}
Let $\left(S,+\right)$ be a dense subsemigroup of $\left(\mathbb{R},+\right)$
containing $0$ such that $\left(S\cap\left(0,1\right),\cdot\right)$is
a subsemigroup of $\left(\left(0,1\right),\cdot\right)$. Let $A$
be a $C_{p}^{0}$ set in $S$ in particular central set near zero,
and $L\in\mathcal{P}_{f}\left(\mathbb{P}\left(S,S\right)\right)$.
Then for each $\delta\in\left(0,1\right)$, there exist functions
$\alpha_{\delta}:\mathcal{P}_{f}\left(\mathcal{T}_{0}\right)\to S$
and $H_{\delta}:\mathcal{P}_{f}\left(\mathcal{T}_{0}\right)\to\mathcal{P}_{f}\left(\mathbb{N}\right)$
such that 
\begin{enumerate}
\item \label{1}$\alpha_{\delta}\left(F\right)<\delta$ for each $F\in\mathcal{P}_{f}\left(\mathcal{T}_{0}\right)$,
\item \label{2}if $F,G\in\mathcal{P}_{f}\left(\mathcal{T}_{0}\right)$
and $F\subset G$, then $\max H_{\delta}\left(F\right)<\min H_{\delta}\left(G\right)$
and
\item \label{3}If $n\in\mathbb{N}$ and $G_{1},G_{2},...,G_{n}\in\mathcal{P}_{f}\left(\mathcal{T}_{0}\right),\,G_{1}\subset G_{2}\subset.....\subset G_{n}$
and $f_{i}\in G_{i},i=1,2,...,n.$ then
\[
\sum_{i=1}^{n}\alpha_{\delta}\left(G_{i}\right)+P\left(\sum_{i=1}^{n}\sum_{t\in H_{\delta}\left(G_{i}\right)}f_{i}\left(t\right)\right)\in A.
\]
 for all $P\in L$.
\end{enumerate}
\end{thm}

\begin{proof}
Choose an idempotent $p\in\mathcal{J}_{p}^{0}$ with $A\in p$. For
$\delta>0$ and $F\in\mathcal{P}_{f}\left(\mathcal{T}_{0}\right)$,
we define $\alpha_{\delta}\left(F\right)\in S$ and $H_{\delta}\left(F\right)\in\mathcal{P}_{f}\left(\mathbb{N}\right)$
witnessing $\left(\ref{1},\ref{2},\ref{3}\right)$ by induction on
$\mid F\mid$.

For the base case of induction, let $F=\left\{ f\right\} $. As $p$
is idempotent, the set $A^{*}=\left\{ x\in A:-x+A\in p\right\} $
belongs to $p$, hence it is a $J_{p}$- set near zero. So for $\delta>0$
there exist $H\in\mathcal{P}_{f}\left(\mathbb{N}\right)$ and $a\in S\cap\left(0,\delta\right)$
such that 

\[
\forall P\in L,\ a+P\left(\sum_{t\in H}f\left(t\right)\right)\in A^{*}.
\]

By setting $\alpha_{\delta}\left(\left\{ f\right\} \right)=a$ and
$H_{\delta}\left(\left\{ f\right\} \right)=H$, conditions $\left(\ref{1},\ref{2},\ref{3}\right)$
are satisfied.

Now assume that $\mid F\mid>1$ and $\alpha_{\delta}\left(G\right)$
and $H_{\delta}\left(G\right)$ have been defined for all proper subsets
$G$ of $F$. Let $K_{\delta}=\bigcup\left\{ H_{\delta}\left(G\right):\emptyset\neq G\subset F\right\} \in\mathcal{P}_{f}\left(\mathbb{N}\right)$,
$m=\max K_{\delta}$ and

Let 
\[
R=\left\{ \begin{array}{c}
\sum_{i=1}^{n}\sum_{t\in H_{\delta}\left(G_{i}\right)}f_{i}\left(t\right)\mid n\in\mathbb{N},\\
\emptyset\neq G_{1}\subsetneq G_{2}\subsetneq\cdots\subsetneq G_{n}\subset F,\\
f_{i}\in G_{i},\forall i=1,2,..,n.
\end{array}\right\} 
\]

\[
M_{\delta}=\left\{ \begin{array}{c}
\sum_{i=1}^{n}\alpha_{\delta}\left(G_{i}\right)+P\left(\sum_{i=1}^{n}\sum_{t\in H_{\delta}\left(G_{i}\right)}f_{i}\left(t\right)\right)\mid n\in\mathbb{N},\\
\emptyset\neq G_{1}\subsetneq G_{2}\subsetneq\cdots\subsetneq G_{n}\subset F,\\
f_{i}\in G_{i},\forall i=1,2,..,n,\,P\in L.
\end{array}\right\} 
\]

Then $R\subseteq S$ and $M_{\delta}$ is finite and by hypothesis
(\ref{3}), $M_{\delta}\subseteq A^{*}$.

Let 
\[
B=A^{*}\cap\left(\bigcap_{x\in M_{\delta}}\left(-x+A^{*}\right)\right)\in p.
\]

For $P\in L$ and $d\in R$, let us define the polynomial 

\[
Q_{P,d}\left(y\right)=P\left(y+d\right)-P\left(d\right)
\]

Since the coefficients of $P$ come from $\mathbb{N}\cup\left\{ 0\right\} $,
$Q_{P,d}\in\mathbb{P}\left(S,S\right)$.

Let $M=L\cup\left\{ Q_{P,d}\mid P\in L\text{ and }d\in R\right\} $.

From Lemma $\ref{min near zero}$, there exists $\gamma\in\mathcal{P}_{f}\left(\mathbb{N}\right)$
with $\min\left(\gamma\right)>m$ and $a\in S\cap\left(0,\delta\right)$
such that 

\[
\forall Q\in M,f\in F\ a+Q\left(f\left(\gamma\right)\right)\in B.
\]
 We set $\alpha_{\delta}\left(F\right)=a<\delta$ and $H_{\delta}\left(F\right)=\gamma$.
So (\ref{1}) is satisfies immediately. Now we are left to verify
conditions (\ref{2}) and (\ref{3}).

Since $\min\left(\gamma\right)>m$ (\ref{2}) is satisfies.

to verify (\ref{3}), let $n\in\mathbb{N}$ and $G_{1},G_{2},...,G_{n}\in\mathcal{P}_{f}\left(\mathcal{T}_{0}\right),\,G_{1}\subset G_{2}\subset.....\subset G_{n}=F$
and $f_{i}\in G_{i},i=1,2,...,n$ and let $P\in L$.

For $n=1$, then $\alpha_{\delta}\left(G_{n}\right)+P\left(\sum_{t\in H_{\delta}\left(G_{n}\right)}f_{n}\left(t\right)\right)=a+P\left(f\left(\gamma\right)\right)\in B\subseteq A^{*}$
.

If $n>1$, then 

\[
\sum_{i=1}^{n}\alpha_{\delta}\left(G_{i}\right)+P\left(\sum_{i=1}^{n}\sum_{t\in H_{\delta}\left(G_{i}\right)}f_{i}\left(t\right)\right)
\]

\[
=\alpha_{\delta}\left(G_{n}\right)+\sum_{i=1}^{n-1}\alpha_{\delta}\left(G_{i}\right)+P\left(\sum_{t\in H_{\delta}\left(G_{n}\right)}f_{i}\left(t\right)+\sum_{i=1}^{n-1}\sum_{t\in H_{\delta}\left(G_{i}\right)}f_{i}\left(t\right)\right)
\]

\[
=a+\sum_{i=1}^{n-1}\alpha_{\delta}\left(G_{i}\right)+P\left(\sum_{t\in\gamma}f_{n}\left(t\right)+\sum_{i=1}^{n-1}\sum_{t\in H_{\delta}\left(G_{i}\right)}f_{i}\left(t\right)\right)
\]

Since $G_{n}=F$ and $\alpha_{\delta}\left(F\right)=a,H_{\delta}\left(G_{n}\right)=\gamma$.

\[
=a+\sum_{i=1}^{n-1}\alpha_{\delta}\left(G_{i}\right)+P\left(\sum_{i=1}^{n-1}\sum_{t\in H_{\delta}\left(G_{i}\right)}f_{i}\left(t\right)\right)+
\]
\[
P\left(\sum_{t\in\gamma}f_{n}\left(t\right)+\sum_{i=1}^{n-1}\sum_{t\in H_{\delta}\left(G_{i}\right)}f_{i}\left(t\right)\right)-P\left(\sum_{i=1}^{n-1}\sum_{t\in H_{\delta}\left(G_{i}\right)}f_{i}\left(t\right)\right)
\]
\[
=a+y+Q_{P,d}\left(\sum_{t\in\gamma}f_{n}\left(t\right)\right)
\]

where $y=\sum_{i=1}^{n-1}\alpha_{\delta}\left(G_{i}\right)+P\left(\sum_{i=1}^{n-1}\sum_{t\in H_{\delta}\left(G_{i}\right)}f_{i}\left(t\right)\right)\in M_{\delta}$,
$d=\sum_{i=1}^{n-1}\sum_{t\in H_{\delta}\left(G_{i}\right)}f_{i}\left(t\right)\in R$
and $P\in L$ so $Q_{P,d}\in M$. 

So we have 
\[
a+Q_{p,d}\left(\sum_{t\in\gamma}f_{n}\left(t\right)\right)\in B\in-y+A^{*}
\]

Therefore 
\[
\sum_{i=1}^{n}\alpha_{\delta}\left(G_{i}\right)+P\left(\sum_{i=1}^{n}\sum_{t\in H_{\delta}\left(G_{i}\right)}f_{i}\left(t\right)\right)\in A^{*}.
\]

This completes the induction argument, hence the proof.
\end{proof}
We can also genralized this theorem along Phulara's way easily 
\begin{thm}
Let $\left(S,+\right)$ be a dense subsemigroup of $\left(\mathbb{R},+\right)$
containing $0$ such that $\left(S\cap\left(0,1\right),\cdot\right)$is
a subsemigroup of $\left(\left(0,1\right),\cdot\right)$. Let $\left\langle C_{n}\right\rangle _{n\in\mathbb{N}}$
be decreasing family of $C_{p}^{0}$ sets in $S$ such that all $C_{i}\in p\in E\left(\mathcal{J}_{P}^{0}\right)$
and $L\in\mathcal{P}_{f}\left(\mathbb{P}\left(S,S\right)\right)$.
Then for each $\delta\in\left(0,1\right)$, there exist functions
$\alpha_{\delta}:\mathcal{P}_{f}\left(\mathcal{T}_{0}\right)\to S$
and $H_{\delta}:\mathcal{P}_{f}\left(\mathcal{T}_{0}\right)\to\mathcal{P}_{f}\left(\mathbb{N}\right)$
such that 
\begin{enumerate}
\item $\alpha_{\delta}\left(F\right)<\delta$ for each $F\in\mathcal{P}_{f}\left(\mathcal{T}_{0}\right)$,
\item if $F,G\in\mathcal{P}_{f}\left(\mathcal{T}_{0}\right)$ and $F\subset G$,
then $\max H_{\delta}\left(F\right)<\min H_{\delta}\left(G\right)$
and
\item If $n\in\mathbb{N}$ and $G_{1},G_{2},...,G_{n}\in\mathcal{P}_{f}\left(\mathcal{T}_{0}\right),\,G_{1}\subset G_{2}\subset.....\subset G_{n}$
and $f_{i}\in G_{i},i=1,2,...,n.$ with $\mid G_{1}\mid=k$ then
\[
\sum_{i\in1}^{n}\alpha_{\delta}\left(G_{i}\right)+P\left(\sum_{i\in1}^{n}\sum_{t\in H_{\delta}\left(G_{i}\right)}f_{i}\left(t\right)\right)\in C_{k}.
\]
\end{enumerate}
\end{thm}

$\vspace{0.3in}$

\textbf{Acknowledgment:}Both the authors acknowledge the Grant CSIR-UGC
NET fellowship with file No. 09/106(0199)/2019-EMR-I and 09/106(0202)/2020-EMR-I
respectively. They also acknowledge the help of their supervisor Prof.
Dibyendu De for valuable suggestions.

$\vspace{0.3in}$

\end{document}